\documentclass[a4paper,12pt]{amsart}
\usepackage[shortlabels]{enumitem}

\numberwithin{equation}{section}

\usepackage{amsfonts}
\usepackage{amsmath, latexsym}
\usepackage{amssymb,verbatim}
\usepackage{color}
\usepackage{euscript}

\newtheorem{thm}{Theorem}
\newtheorem{lem}[thm]{Lemma}
\newtheorem{prop}[thm]{Proposition}
\newtheorem{cor}[thm]{Corollary}

\newenvironment{proof*}{\vskip 2mm\noindent {}}{\hfill $\Box$ \vskip 2mm}

\newcommand{\ba}{\begin{align*}}
\newcommand{\ea}{\end{align*}}

\newcommand{\R}{\mathbb{R}}
\newcommand{\C}{\mathbb{C}}
\newcommand{\Z}{\mathbb{Z}}

\newcommand{\eps}{\varepsilon}

\newcommand{\f}{\varphi}
\renewcommand\d{\delta}

\newcommand{\be}{\begin{equation}}
\newcommand{\ee}{\end{equation}}

\newcommand{\dist}{\mathop{\rm dist \,}}

\title{Boundary regularity for the distance functions, and the eikonal equation}

\author{Nikolai Nikolov}
\address{N. Nikolov\\Institute of Mathematics and Informatics\\Bulgarian Academy
of Sciences\\Acad. G. Bonchev 8, 1113 Sofia, Bulgaria
\vspace{1mm}
\newline Faculty of Information Sciences\\
State University of Library Studies and Information Technologies\\
Shipchenski prohod 69A, 1574 Sofia,
Bulgaria}

\email{nik@math.bas.bg}

\author{Pascal J. Thomas}
\address{P.J. Thomas\\
Institut de Math\'ematiques de Toulouse; UMR5219 \\
Universit\'e de Toulouse; CNRS \\
UPS, F-31062 Toulouse Cedex 9, France}

\email{pascal.thomas@math.univ-toulouse.fr}

\thanks{The  first named author was partially supported by the Bulgarian National
Science Fund, Ministry of Education and Science of Bulgaria under contract KP-06-N52/3.}

\keywords{eikonal equation, distance function, defining function, gain in regularity}

\subjclass[2020]{35F20, 35F21, 35B65}

\begin{document}

\begin{abstract}
We  study the gain in regularity of the distance to the boundary of a domain in $\R^m$.
In particular,
we show that if the signed distance function happens to be merely differentiable
in a neighborhood of a boundary point, it and the boundary have to be $\mathcal C^{1,1}$
regular. Conversely, we study the regularity of the distance function under
regularity hypotheses of the boundary.
Along the way, we point out that any solution to the eikonal equation, differentiable everywhere in
a domain of the Euclidean space, admits a gradient which is locally Lipschitz.
\end{abstract}

\maketitle

\section{Introduction}
\label{intro}

Let $D\subsetneq \R^m$ be a domain ($m\ge 2$). The distance to its
boundary is denoted $\d_D(x):= \min_{y\in\partial D} |x-y|$.
The \emph{signed distance function} to $\partial D$ is defined by
 $d_D:=\d_D$ on $D$ and $d_D:=-\d_D$ on $\R^m\setminus D.$

Throughout this note, we write $\langle\cdot,\cdot\rangle$ for the usual Euclidean
inner product in $\R^m$, and $|x|:=\langle x,x\rangle^{1/2}$ for the Euclidean norm.

Our main goal is to prove that the
signed distance function has a bootstrap property: if $d_D$ is
differentiable in an open set $U$, it must be $\mathcal C^{1,1}$ regular
(Corollary \ref{main}). This is a generalization of results obtained
in a succession of previous works, notably \cite{Fed, Luc, KP, CC}.

We achieve this in Section \ref{boot} by proving that $d_D$ satisfies the \emph{eikonal equation}
 $|\nabla d_D(x)|=1$ where it is differentiable. This is easy and well-known away from $\partial D$,
 and we show how to extend it to the boundary in Proposition \ref{bdrypt}.
 We then bring to bear previous works about the eikonal equation, summed up
 in Proposition \ref{c11}.

 While the eikonal equation may seem a more general hypothesis, Caffarelli and Crandall
 \cite{CC} proved that, up to additive constant, solutions of the eikonal equation
 are actually distance functions \cite[(1.9) and Proposition 4.4]{CC}, a result recalled here in
 Lemma \ref{eikondist}.  Where it is defined, $\nabla d_D(x)$ is a divergence-free 
 unit vector field. Gains in regularity also occur in this more
general case of vector fields; for the case $m=2$, \cite[Theorem 1]{Ign1} shows that 
if $D$ is a domain in $\R^2$,
any such vector field which is in the Sobolev space $W^{1/p, p} (D)$
for some $p\in [1,2]$ must be locally Lipschitz continuous
inside $D$ (which means locally the gradient of a $\mathcal C^{1,1}$
potential) except at a locally finite number of singular points. One can
also find in \cite{Ign1} a wealth of examples on those topics, including
cases where the $\mathcal C^{1,1}$ potential cannot be of class $\mathcal C^2$, and
a simply connected domain which allows for an infinite number of said singular points. 

Then in Section \ref{bdry}, we turn to results about
the distance function under hypotheses on the regularity of $\partial D$.
To conclude, we prove Proposition \ref{c1}, a more precise estimate about the variation of
$d_D$.

\section{Bootstrap properties of $d_D$}
\label{boot}

In general, $d_D$ is a $1$-Lipschitz function. It has been known for a long
time that it inherits much of the smoothness of $\partial D$, see \cite{KP} and the references therein.
Precisely, when $\partial D$ is $\mathcal C^k$-smooth, $k\ge 2$,
so is $d_D$ in a neighborhood of $\partial D$;
the same is true for $k=1$ under the additional hypothesis that $\partial D$ is of
\emph{positive reach}, that is, that
$x$ admits a unique projection to $\partial D$ whenever the distance from $x$ to $\partial D$ is less
than a certain uniform positive number \cite{Fed}.

Recall the following basic fact
(see e.g. \cite[Theorem 4.8]{Fed}).

\begin{prop}
\label{diffd}
Let $D\subsetneq \R^m$ be an open set. Then $d_D$ is differentiable at $x\not\in\partial D$
exactly when $x$ admits a unique projection $\pi(x)$ to $\partial D.$ In that case,
$\displaystyle\nabla d_D(x)=\frac{x-\pi(x)}{d(x)}.$
\end{prop}

In particular, whenever $d_D$ is differentiable at $x \notin \partial D$,
$d_D$ satisfies the eikonal equation.
This equation has been much studied, see e.g.  \cite{CC}, \cite{DLI}.
It turns out to extend to the boundary when $\nabla d_D$ still makes sense in a neighborhood of a point; in fact it is enough to have differentiability at a boundary point and in
a one-sided neighborhood of it.

\begin{prop}
\label{bdrypt}
Let $D\subsetneq \R^m$ be an open set.
 If $d_D$ is differentiable at $p \in \partial D$ and in $U\cap D$, for
a neighborhood $U$ of $p$, then $|\nabla d_D(p)| =1$.
\end{prop}

\begin{proof}
At first, suppose only that $d_D$ is differentiable at a point $p\in \partial D$. 

We then prove that if $\nabla d_D(p)\neq 0$, then $|\nabla d_D(p)|=1$. 
Since $d_D$ is $1$-Lipschitz, $|\nabla d_D(p)|\le 1$. Take coordinates
in an orthonormal basis such that $p=0$ and $\nabla d_D(p)=\alpha e_1
=(\alpha,0)$, $\alpha \in (0,1]$.

Suppose $\alpha<1$.
Then for  any small enough  $x_1>0$, $d_D(x_1 e_1) >0$ so $x_1\in D$ and 
$|\pi(x_1 e_1)-x_1 e_1|< (1+\alpha) x_1/2$, where $\pi(x_1 e_1)$ is any
of the closest points to $x_1 e_1$ on $\partial D$. Let $q\in D$ be such that 
$|\pi(x_1 e_1)-q|< \alpha(1-\alpha)x_1/4$;
we must have
\[
q_1> x_1-\frac{1+\alpha}2 x_1 - \alpha \frac{1-\alpha}4 x_1 =
(1-\frac\alpha2)\frac{1-\alpha}2 x_1,
\]
and
\[
|q|\le x_1 +   \frac{1+\alpha}2 x_1 + \alpha \frac{1-\alpha}4 x_1=Cx_1.
\]
Finally
\[
\frac\alpha4 (1-\alpha)x_1 > d_D(q) = \alpha q_1 + o(|q|)
\ge \alpha (1-\frac\alpha2)\frac{1-\alpha}2 x_1 + o(x_1),
\]
a contradiction, so $|\nabla d_D(0)|=1$.

Now we are reduced to excluding the case $|\nabla d_D(p)| =0$ when $d_D$
is differentiable on $D\cap U$.  

Again, we may assume $p=0$. If $|\nabla d_D(0)| =0$, then for any $\eps>0$, there exists $\delta_\eps >0$
such that if $|x|<\d_\eps$, then $|d_D(x)| < \eps |x|$.

For some $x\in D$ such that $d_D$  is differentiable in a neighborhood of $x$, and therefore
admits a continuous gradient there by the formula in Proposition \ref{diffd}, we 
consider the integral curve of the vector field $\nabla d_D$ with starting point at $x$, i.e.
the map $\chi: [0, T)\longrightarrow D$ such that $\chi'(t)= \nabla d_D(\chi(t))$, $\chi(0)=x$
(this is known as a \emph{characteristic} of the function $d_D$ and will turn out to be an affine map \cite[Lemma 2.2]{CC}, but we shall not need this fact at this point).

Notice that $d_D(x)>0$ by hypothesis and $d_D(\chi(t))$ can only increase as $t$ grows,
as long as it is defined, so $\chi(t)\in D$ and the curve must meet a point
of non-differentiability before it could escape $D$. Let
\[
t_x:= \sup\left\{t: d_D \mbox{ is differentiable in a neighborhood of }\chi(s), 0\le s \le t \right\}.
\]

By construction, for any $\eta>0$, there is a point $x' \in D$ where $d_D$ is not differentiable
and $|x'-\chi(t_x)|<\eta$. By the differential equation,
for $0\le s\le t_x$,
$$
d_D(\chi(s))= d_D(x)+\int_0^{s} \langle \nabla d_D (\chi(t)), \chi'(t)\rangle dt
= d_D(x) + \int_0^{s} | \nabla d_D (\chi(t))|^2  dt =d_D(x)+ s.
$$

We claim that for any $\eps>0$, there exists $y\in B(0,\eps)\cap D$ such that $d_D$
is not differentiable at $y$.
For any $\eps \in (0,\frac14)$, choose $\delta $
small enough so that $\d < \eps/4$, $2\d < \d_{\eps}$.  Let $x\in B(0,\d)\cap D$.
If $d_D$ is not differentiable at $x$, we can take $y=x$.

If it is, choose $x'$
as above with $\eta:=\eps/10$. We claim that $t_x \le  |x|$. Indeed, if
$t_x > |x|$, then for  $s=   |x|$, $d_D(\chi(s))\ge d_D(x)+|x|$.
On the other hand, $|\chi(s)| \le 2 |x| < \d_\eps$, so that
$d_D(\chi(s))< \eps |\chi(s)| < \frac12 |x|$, a contradiction.

Finally $|x'|\le |x| + t_x + \eta < \eps$, and we can take $y=x'$.
\end{proof}

Notice that the case where $\nabla d_D (p)=0$ can occur.  We give an example in $\R^2$,
which we identify with $\C$ in order to use polar coordinates.

Let $f: [0,\infty) \longrightarrow (0,\infty)$ be a continuous strictly
decreasing function satisfying 
\newline $\lim_{\theta\to\infty} f(\theta) = 0$. 
Let 
\[
\Omega:= \{ re^{i\theta}: 0<\theta, f(\theta+\pi)<r<f(\theta)\}. 
\]
The domain $\Omega$ looks like a thickened spiral and 
for any $\alpha$, 
\[
\Omega \cap \R_+ e^{i\alpha} = \{ re^{i\alpha}: 
f (\alpha+(2k+1)\pi) < r < f (\alpha+2k\pi), k\in \Z\}.
\]

\begin{prop}
\label{gradzero}
If $\lim_{\theta\to\infty} \frac{f(\theta)}{f(\theta+\pi)}=1$,
then $\Omega$ is a connected, simply connected domain such that
at the point $0\in \partial \Omega$, the signed distance function $d_\Omega$
is differentiable and has gradient equal to $0$.
\end{prop}

\begin{proof}
The domain $\Omega_0:=\{ (r,\theta) \in (0,\infty)\times (0,\infty): 0<\theta, f(\theta+\pi)<r<f(\theta)\}$ is clearly connected and simply connected.

The map $(r,\theta) \mapsto re^{i\theta}$ is continuous, so $\Omega$
is connected.  For a point $re^{i\theta} \in \Omega$, the condition 
$f(\theta+\pi)<r<f(\theta)$ means that there is a single possible choice 
of $\theta$, therefore the inverse map is well defined and continuous, so 
$\Omega$ is homeomorphic to $\Omega_0$.

Let $z=re^{i\theta} \in \Omega$. Considering the line segment 
$[e^{i\theta} f(\theta+\pi) ; e^{i\theta} f(\theta)]$ we see that 
$d_\Omega(z)\le   f(\theta)-f(\theta+\pi)$ and $|z|\ge  f(\theta+\pi)$.
Therefore 
\[
0\le \frac{d_\Omega(z)}{|z|}
 \le \frac{f(\theta)}{f(\theta+\pi)} -1 \to 0,
\]
as $|z|\to 0$ (which implies $\theta \to \infty$). 

For points outside of $\Omega$, the proof is similar.
\end{proof}

In the above example, between the two spirals that constitute $\partial \Omega$,
there is a spiral of points of $\Omega$ which are equidistant from the two components and
where $d_D$ is not differentiable.

The following is more or less implicit in \cite{CC}.

\begin{prop}
\label{c11}
Let $\Omega \subset \R^m$ be an open set and
$\varphi: \Omega \rightarrow \mathbb{R}$ be an everywhere differentiable
solution of the eikonal equation $|\nabla \varphi|=1$ in $\Omega$. Then
$\varphi\in \mathcal C^{1,1}(\Omega)$, i.e. $\nabla \varphi$ is locally Lipschitz in $\Omega$.

As a consequence, if  $d_D$ is differentiable in a neighborhood $U$ of
$\partial D,$ and $|\nabla d_D(x)|=1$ for $x\in U\cap\partial D$, then
$d_D\in \mathcal C^{1,1}(U)$.
\end{prop}

\begin{proof}
We outline how to use the results of Caffarelli and Crandall \cite{CC} to prove
this proposition in our more special case (they consider a more general family of
norm and functions with a possible small singularity set).

By \cite[Lemma 2.2]{CC}, if $a<0<b$ and $x+t\nabla \varphi(x) \in U$ for $a<t<b$,
then for those values of $t$, $\nabla \varphi(x+t\nabla \varphi(x)) = \nabla \varphi(x)$
and $ \varphi(x+t\nabla \varphi(x)) = t$. Moreover, $\varphi \in \mathcal C^1(U)$.

The straight lines $\{ x+t\nabla \f (x), t \in \R\}$ are called \emph{characteristics}
of $\f$. For any $a \in \f(U)$, we denote the level set by $S_a:= \{x\in U: \f(x)=a\}$.

The following is a direct consequence of \cite[Proposition 4.4]{CC} and its proof.

\begin{lem}
\label{eikondist}
Let $V$ be an open set of the form
\[
\{ x+t\nabla \f (x), t \in (a(x),b(x)), x \in S_a\cap V\},
\]
where $a \in \f(U)$ and $a(x)<0<b(x)$.  Then for any $y\in V$, $d_{S_a}(y)= \f(y)-a$.
\end{lem}

In order to apply it to our situation, we need to see that any $x\in U$ admits a
neighborhood $V$ of the form above.
Let $x_0 \in U$.  Choose affine local coordinates $(x', x^m) \in \R^{m-1}\times \R$
 so that $x_0=0$ and $\nabla\f(x_0)=(0,\dots,0,1)$. Then $\Phi(x',x^m)= (x',\f(x))$
 is a local $\mathcal C^1$-diffeomorphism. In particular, for $r>0$
 small enough and $|x'|\le r$, there exists a unique $x^m$ such that
$\f(x',x^m)= \f(0)$.

We can choose $\d$ small enough so that
$$
U_r:= \left\{ x+t\nabla \f (x): x\in S_{\f(x_0)}, |x'| < r, |t| <r \right\}
=\Phi^{-1}\left( B^{m-1}(0,r) \times (-r,r) \right)
$$
is a open neighborhood of $x_0$ contained in $U$ for $r\le 3\delta$,
made up of a disjoint union of
segments of characteristics.  Let $a:= \f(x_0)-2\delta$, then if $x\in U_\d$,
we have $\f(x)= \dist (x, S_a) - 2 \d +f(x_0)$, so $\nabla \f(x)= \nabla d_{S_a} (x)$.

Using Proposition \ref{diffd}, for $x,y\in U_\d$,
\begin{multline*}
\nabla \f(x) - \nabla \f(y) =
 \frac{x-\pi_{S_a}(x)}{d_{S_a}(x)} - \frac{y-\pi_{S_a}(y)}{d_{S_a}(y)}\\
 = \frac1{d_{S_a}(x)} \left(  x-y - (\pi_{S_a}(x)-\pi_{S_a}(y))\right)
 - (x - \pi_{S_a}(x)) \frac{d_{S_a}(x)-d_{S_a}(y)}{d_{S_a}(x)d_{S_a}(y)},
\end{multline*}
so that, since $d_{S_a}(x), d_{S_a}(y) \ge \delta$, then
$|\nabla \f(x) - \nabla \f(y) | \le 2 \d^{-1} |x-y| + \d^{-1} |x-y|.$
\end{proof}

\begin{cor}
\label{main}
If $d_D$ is differentiable in a neighborhood $U$ of $p\in \partial D$, then
$d_D\in \mathcal C^{1,1}(U)$.
\end{cor}
\begin{proof}
By Proposition \ref{diffd}, $|\nabla d_D|=1$ on $U\setminus \partial D$.
By Proposition \ref{bdrypt}, $|\nabla d_D|=1$ on $U\cap \partial D$. Then Proposition \ref{c11}
applies.
\end{proof}

\noindent{\bf Remark.} If we already know that  $d_D \in \mathcal C^1(U)$, for $U$ a neighborhood  of
$\partial D,$ then the hypotheses of the second part of Proposition \ref{c11}
are satisfied.  With that hypothesis,
$\partial D$ is $\mathcal C^1$-smooth (since $|\nabla d_D|=1$ on $U$ by
 continuity). In this case, as pointed out in \cite[p.~120]{KP}, K.~Lucas's work
\cite[Section 2]{Luc} implies that $\partial D$ is $\mathcal C^{1,1}$-smooth
(see also \cite[Theorem 4.18]{Fed})
and hence, by Proposition \ref{c1} below, $d_D$ is $\mathcal C^{1,1}$-smooth near $\partial D.$
But that proof is rather more roundabout than the arguments from \cite{CC}.
\smallskip

\section{Boundary regularity assumptions}
\label{bdry}

We now turn to weaker hypotheses, involving only the regularity of $\partial D$ itself.

\begin{prop}\label{c4} Let $p$ be a $\mathcal C^1$-smooth boundary point of a domain $D$ in $\R^m.$
Then $d_D$ is differentiable at $p$ and $\nabla d_D(p)$ is the inner unit normal vector to $\partial D$ at $p.$
\end{prop}

This does not necessarily extend to any neighborhood.
For $0<\alpha<1$ the domain $D_\alpha=\{x\in\R^2:x_1>|x_2|^{1+\alpha}\}$ is $\mathcal C^{1,\alpha}$- but not $\mathcal C^{1,1}$-smooth.
By Propositions \ref{diffd} and \ref{c4}, the function $d_D$ is non-differentiable at $x$ exactly when
$x=(x_1,0)$ with $x_1>0.$

\begin{proof} We may assume that that $D=\{x_1>f(x')\}$ near $p=0,$ where
$x=(x_1,x')\in \R\times \R^{m-1}$ and $f(x')=o(|x'|).$

Let $x\in D.$ Since $\tilde x=(f(x'),x')\in\partial D,$ then
$$
d_D(x)-x_1\le|x-\tilde x|-x_1=-f(x').
$$
On the other hand, let $\hat x\in\partial D$ be such that $d_D(x)=|x-\hat x|.$ Then
$$d_D(x)-x_1\ge|x_1-\hat x_1|-x_1\ge-|\hat x_1|=-|f(\hat x')|.$$
Since $|\hat x|\le|x|+d_D(x)\le 2|x|,$ it follows that
$d(x)-x_1=o(|x'|).$

Similar arguments imply the same for $x\not\in D$ which completes the proof.
\end{proof}

We now give a sufficient condition for the $\mathcal C^{1,1}$ smoothness of $d_D$
in terms of conditions on $\partial D$ only, along with a more precise estimate
of its second order variation. We need some notation.  For $p\in \partial D$ a point of the boundary
where the distance function $d_D$ is differentiable,
let $n_p:= \nabla d_D (p)$ denote the inner normal vector. We want to study 
the variation of the inner normal, and set 
$$
\chi_{D,p}:=\limsup_{\partial D\ni p',p''\to p}\frac{|n_{p'}-n_{p''}|}{|p'-p''|}.
$$

\begin{prop}\label{c1} Let $p$ be a $\mathcal C^{1,1}$-smooth boundary point of a domain
$D$ in $\R^m.$ Then $d_D$ is $\mathcal C^{1,1}$-smooth near $p.$ Moreover, for
$x, y$ near $p$ one has that
$$
|d_D(x)-d_D(y)-\langle\nabla d_D(x),x-y\rangle|\le(\chi_{D,p}/2+o(1))(|x-y|^2-(d_D(x)-d_D(y))^2).
$$
\end{prop}

\begin{proof}
Recall that when $q\in \partial D$, $\nabla d_D(q)=n_q$, the unit inner normal vector to $\partial D$ at
$q\in\partial D$. Set $\tilde x$ to be the projection of $x$ near $p$ on $\partial D$
(it is unique by $\mathcal C^{1,1}$-smoothness). Then
\begin{multline*}
\langle n_{\tilde x},x-y\rangle=\langle n_{\tilde x},\tilde x+d_D(x)n_{\tilde x}-\tilde y-d_D(y)n_{\tilde y}\rangle=
\\
d_D(x)-d_D(y)+\langle n_{\tilde x},\tilde x-\tilde y\rangle+(1-\langle n_{\tilde x},n_{\tilde y}\rangle)d_D(y).
\end{multline*}
Let $\chi_0=\chi_{D,p}.$
Since $|n_{\tilde x}-n_{\tilde y}|\le(\chi_0+o(1))|\tilde x-\tilde y|,$ it follows that
$$
1-\langle n_{\tilde x},n_{\tilde y}\rangle\le(\chi_0^2/2+o(1))|\tilde x-\tilde y|^2.
$$

To estimate $\langle n_{\tilde x},\tilde x-\tilde y\rangle,$
we may assume that $p=0$ and
$\partial D$ near $0$ is given by $u_1=f(u'),$ where $u':=(u_2, \dots, u_m)$,
 for $|u'|<\varepsilon_0,$ with $f(0)=0,$ and $\nabla f(0)=0.$

For $\tilde x= (f (\tilde x'), \tilde x')$,
$(1,-\nabla f (\tilde x'))=  \sqrt{1+|\nabla f(\tilde x')|^2} \, n_{\tilde x}
=(1+o(1)) n_{\tilde x}$, so
\begin{equation}
\label{nablest}
|\nabla f(u')-\nabla f(v')|\le(\chi_0+o(1))|u'-v'|.
\end{equation}

Then
$$
\sqrt{1+|\nabla f(\tilde x')|^2}\langle n_{\tilde x},\tilde x-\tilde y\rangle=
f(\tilde x')-f(\tilde y')-\langle\nabla f(\tilde x'),\tilde x'-\tilde y'\rangle.
$$
Writing $g(t):= f((1-t) \tilde x' + t \tilde y')$, $g\in \mathcal C^{1,1}$, and we
need to estimate  $|g(1)-g(0)-g'(0)|$; by Taylor's formula with integral remainder,
it is bounded by \[
\frac12 |\tilde x'-\tilde y'|^2 \sup_{[0,1]} |g''| = \frac12(\chi_0+o(1))|\tilde x'-\tilde y'|^2,
\]
by \eqref{nablest}. Finally
\[
|\langle n_{\tilde x},\tilde x-\tilde y\rangle | \le (\chi_0/2+o(1)) |\tilde x'-\tilde y'|^2
\le (\chi_0/2+o(1)) |\tilde x-\tilde y|^2,
\]
thus
\begin{multline*}
\left| d_D(x)-d_D(y)-\langle n_{\tilde x},x-y\rangle \right| =
\left| \langle n_{\tilde x},\tilde y-\tilde x\rangle - (1-\langle n_{\tilde x},n_{\tilde y}\rangle)d_D(y) \right|
\le (\chi_0/2+o(1)) |\tilde x-\tilde y|^2.
\end{multline*}

It remains to compare $|\tilde x-\tilde y|^2$ with $|x-y|^2$. But
\begin{multline*}
|x-y|^2=
\left| (\tilde x-\tilde y) + (d_D(x)-d_D(y))  n_{\tilde x} + d_D(y) (n_{\tilde x}-n_{\tilde y}) \right|^2
\\=|\tilde x-\tilde y|^2+(d_D(x)-d_D(y))^2
+2\big[d_D(x)\langle n_{\tilde x},\tilde x-\tilde y\rangle+d_D(y)\langle n_{\tilde y},
\tilde y-\tilde x\rangle+d_D(x)d_D(y)(1-\langle n_{\tilde x},n_{\tilde y}\rangle) \big]
\\
=(1+o(1))|\tilde x-\tilde y|^2+(d_D(x)-d_D(y))^2.
\end{multline*}
So
$$
\left| d_D(x)-d_D(y)-\langle n_{\tilde x},x-y\rangle \right|\le(\chi_0/2+o(1))(|x-y|^2-(d_D(x)-d_D(y))^2).$$
Hence $d_D$ is $\mathcal C^{1,1}$-smooth near $p$ and $\nabla d_D(x)=n_{\tilde x}.$
\end{proof}

\noindent{\bf Acknowledgements.} We wish to thank our colleague Radu Ignat for useful 
discussions about the eikonal equation. 
He also has a direct proof for the $\mathcal C^{1,1}$ interior regularity 
of solutions of the eikonal equations in any dimension \cite{Ign2}.

\end{document}